\documentclass[11pt]{article}
\usepackage{xcolor}
\usepackage{hyperref,graphicx,amssymb,amsthm,amsmath,framed,appendix,fullpage,tikz,tikz-cd,array,verbatim,multicol,indentfirst,enumitem,lscape,arydshln,pgfplots,esint}
\usepackage[normalem]{ulem}
\usepackage{cleveref}
\pgfplotsset{compat=1.12}
\usetikzlibrary{fit,shapes,arrows}

\newtheoremstyle{dotless}{}{}{}{}{\bfseries}{}{ }{}

\theoremstyle{definition}

\newtheorem{sthm}{Theorem}[section]

\newtheorem{slem}[sthm]{Lemma}
\newtheorem{scor}[sthm]{Corollary}
\newtheorem{sprop}[sthm]{Proposition}

\newtheorem{srmk}[sthm]{Remark}

\newtheorem{sex}[sthm]{Example}
\newtheorem{sqn}[sthm]{Question}

\newtheorem{sdef}[sthm]{Definition}

\theoremstyle{dotless}
\newtheorem*{prob2}{Problem}

\newcommand{\mf}[1]{\mathfrak{#1}}

\newcommand{\mb}[1]{\mathbb{#1}}

\newcommand{\on}[1]{\operatorname{#1}}

\newcommand{\im}{\on{im}}

\newcommand{\spec}{\on{Spec}}

\renewcommand{\phi}{\varphi}

\renewcommand{\dot}{\bullet}

\newcommand{\ann}{\on{Ann}}

\newcommand{\mbn}{\mb{N}}

\newcommand{\mbz}{\mb{Z}}

\newcommand{\ass}{\on{Ass}}

\newcommand{\mfp}{\mathfrak{p}}
\newcommand{\mfq}{\mathfrak{q}}
\newcommand{\mfm}{\mathfrak{m}}

\newcommand{\mfb}{\mf{b}}

\tikzset{n/.style={draw,circle,fill=white}}
\tikzset{w/.style={rectangle,fill=white}}
\tikzset{lab/.style={rectangle,draw}}
\tikzset{d/.style={draw,circle,fill,scale=.5}}
\tikzset{od/.style={draw,circle,scale=.5}}

\usepackage[symbol*,perpage]{footmisc}

\usepackage[noadjust]{cite}
\usepackage{cleveref}
\newcommand{\fbp}[1]{\left[ #1\right]}
\newcommand{\hsl}{\on{HSL}}
\newcommand{\hslr}{\on{HSL}_R}
\newcommand{\fte}{\on{Fte}}
\newcommand{\specnot}{\spec^\circ}
\newcommand{\assnot}{\ass^\circ}

\begin{document}
\title{A sufficient condition for finiteness of Frobenius test exponents}
\author{Kyle Maddox}
\date{18 April, 2019}
\maketitle\vspace*{-24pt}
\begin{abstract}
\noindent The Frobenius test exponent $\fte(R)$ of a local ring $(R,\mfm)$ of prime characteristic $p>0$ is the smallest $e_0 \in \mbn$ such that for every ideal $\mfq$ generated by a (full) system of parameters, the Frobenius closure $\mfq^F$ has $(\mfq^F)^{\fbp{p^{e_0}}}=\mfq^{\fbp{p^{e_0}}}$. We establish a sufficient condition for $\fte(R)<\infty$ and use it to show that if $R$ is such that the Frobenius closure of the zero submodule in the lower local cohomology modules has finite colength, i.e. $H^j_\mfm(R)/0^F_{H^j_\mfm(R)}$ is finite length for $0 \le j < \dim(R)$, then $\fte(R)<\infty$. 
\end{abstract}
\section*{Introduction}

Let $(R,\mfm)$ be a commutative Noetherian local ring of characteristic $p>0$. The Frobenius closure of an ideal $I\subset R$ is defined to be the ideal $I^F = \{ x \in R \mid x^{p^e} \in I^{\fbp{p^e}} \text{ for some } e\in \mbn\}$. In general, computing the Frobenius closure of an ideal should be expected to be difficult, because we must check infinitely many equations for every element of the ring. However, Frobenius bracket powers are much simpler to compute. Since $R$ is Noetherian we must have an $e_0 \in \mbn$ such that $(I^F)^{\fbp{p^e}} = I^{\fbp{p^e}}$ for all $e \ge e_0$, so we can simply check one equation -- $x \in I^F$ if and only if $x^{p^{e_0}} \in I^{\fbp{p^{e_0}}}$. However, computing the required $e_0$ for each $I$ might also be difficult, so it would be desirable to get uniform bounds for all $I$ depending only on the ring.

One cannot expect uniform behavior like this even in nice rings -- Brenner \cite{B} showed that in a two-dimensional domain standard graded over a field we can have a sequence of ideals where the required exponent tends to infinity. However, some finiteness results are known if we restrict to the class of parameter ideals -- the \textbf{Frobenius test exponent} for (parameter ideals of) $R$ is the smallest $e_0$ such that for any $\mfq\subset R$ a parameter ideal, $(\mfq^F)^{\fbp{p^{e_0}}} = \mfq^{\fbp{p^{e_0}}}$. 

Katzman and Sharp \cite{KS} showed that $\fte(R)<\infty$ if $R$ is Cohen-Macaulay. The same year, Huneke, Katzman, Sharp, and Yao \cite{HKSY} showed $\fte(R)<\infty$ if $R$ is generalized Cohen-Macaualy using some very involved techniques. More recently, Quy \cite{Q} introduced a new technique which vastly simplified the proof for generalized Cohen-Macaulay rings and also showed F-nilpotent rings have finite Frobenius test exponent. Quy's proofs suggest a sufficient condition for finiteness of the Frobenius test exponent (see \Cref{the condition}), and we can extend his techniques to show that a new class of F-singularity which we call \textbf{generalized weakly F-nilpotent rings} (see \Cref{defn gwfnilp} and \Cref{gwfncase}) also have finite Frobenius test exponent.\\

\noindent \textbf{Acknowledgments} \textit{The author would like to thank his advisor Ian Aberbach for suggesting generalized F-nilpotent rings as a case study for the sufficient condition and countless corrections and ideas. Also, the author would like to thank Thomas Polstra and Pham Hung Quy for looking over initial drafts of this paper and providing a multitude of helpful comments, as well as the anonymous referee for many useful suggestions.} \newpage

\noindent \textbf{Notation and conventions:} Throughout, $(R,\mfm)$ will be a Noetherian local ring of dimension $d$ and of prime characteristic $p>0$. By a parameter ideal, we will mean an ideal generated by a full system of $d$ parameters. Write $\specnot(R)=\spec(R)\setminus \{\mfm\}$ and for any subset $X\subset \spec(R)$, write $X^\circ = X \cap \specnot(R)$. In particular, $\assnot_R(M)=\ass_R(M)\cap \specnot(R)$. The set $\mbn$ contains $0$ and $\mbz_+$ will be used for the set of positive integers.

If $x_1,\cdots,x_t$ is an (ordered) sequence of elements of $R$, write $\uline{x}=x_1,\cdots,x_t$ for the list of elements. Given $\uline{x}=x_1,\cdots,x_t$ and a sequence $n_1,\cdots, n_t \in \mbz_+$, write $\uline{x}^{\uline{n}} = x_1^{n_1},\cdots,x_t^{n_t}$ for the new sequence obtained by taking powers. In particular, if $n \in \mbz_+$ then $\uline{x}^n = x_1^n,\cdots,x_t^n$. If a sequence $\uline{x}=x_1,\cdots,x_t$ is given and $J=(\uline{x})$, then we write $J_i=(x_1,\cdots,x_i)$  for $1 \le i \le t$. Set $J_0=0$. 

\section{Background}
\subsection{Frobenius closure of an ideal and Frobenius test exponents}

Throughout this subsection, let $R$ be Noetherian and of prime characteristic $p>0$.

\begin{sdef}\label{frob closure defn}
Let $I\subset R$ be an ideal. The \textbf{Frobenius closure of $I$} is the ideal: \[ I^F=\left\lbrace x \in R \mid x^{p^e} \in I^{\fbp{p^e}} \text{ for all } e \gg 0\right\rbrace. \] Call $I$ \textbf{Frobenius closed} if $I^F=I$. 
\end{sdef}

Clearly $I\subset I^F$, and $I^F$ is Frobenius closed. Furthermore, if $I\subset J$, then $I^F \subset J^F$ as for any $e\in \mbn$, $I^{\fbp{p^e}}\subset J^{\fbp{p^e}}$.

\begin{sdef}\label{ftedefn}
Let $I\subset R$ be an ideal. Since $R$ is Noetherian, there is an $e_0 \in \mbn$ such that $(I^F)^{\fbp{p^e}}= I^{\fbp{p^e}}$ for all $e \ge e_0$. Call the smallest such exponent the \textbf{Frobenius test exponent for $I$}, and denote it $\fte(I)$.
\end{sdef}

It is natural to desire an $e_0 \in \mbn$ such that for all ideals $I\subset R$, $\fte(I)\le e_0$. Unfortunately, this is too much to ask even for nice, low dimensional rings -- Brenner \cite{B} gives a counterexample in a two-dimensional normal graded domain. 

There are cases where uniform bounds on the required exponent for \textit{parameter ideals} are known. In particular, Katzmann and Sharp \cite{KS} showed that if $(R,\mfm)$ is a Cohen-Macaulay local ring of dimension d and prime characteristic $p>0$,  and $\mfq\subset R$ a parameter ideal, then $\fte(\mfq) \le \hsl(H^d_\mfm(R))$ (see \Cref{hsl numbers defn}).  

\begin{sdef}
We define the \textbf{Frobenius test exponent (for parameter ideals) of $R$} to be: \[ \fte(R) = \inf \left\lbrace e \in \mbn \mid (\mfq^F)^{\fbp{p^e}} = \mfq^{\fbp{p^e}} \text{ for all parameter ideals } \mfq \subset R \right\rbrace \in \mbn \cup \{\infty\}.\]
\end{sdef}

This number is a coarse measure of singularity in characteristic $p$ -- if $R$ is regular than $\fte(R)=0$, however there are non-regular rings with $\fte(R)=0$ -- for instance, F-injective Cohen-Macaulay rings or F-pure rings. The authors of \cite{QS} define $R$ to be \textbf{parameter F-closed} when $\fte(R)=0$, and so $\fte(R)$ is a measure of how close $R$ is to being parameter F-closed.

\begin{sqn} 
For which local rings $R$ of prime characteristic $p>0$ is $\fte(R)<\infty$?
\end{sqn}

As mentioned in the introduction, the following cases were known previously.

\begin{itemize}
\item \cite{KS} Cohen-Macaulay rings
\item \cite{HKSY} Generalized Cohen-Macaulay rings (see \Cref{gcm defn})
\item \cite{Q} Weakly F-nilpotent rings (see \Cref{wfn defn})
\end{itemize}

In Section 3, we will extend this list to include a new class of F-singularity, called generalized weakly F-nilpotent rings (\Cref{gwfncase}) and recapture the previous cases as corollaries. 

\subsection{Filter regular sequences}
In this subsection, the assumption that $R$ is of prime characteristic is unnecessary.

We will regularly use the notion of filter regular sequences throughout this paper, so we cover some basic properties here. Filter regular sequences are a generalization of regular sequences, and we will see that every parameter ideal can be generated by a filter regular system of parameters. This allows many proofs for Cohen-Macaulay rings to work (with minor modifications) in essentially any local ring. We will use this to create a powerful long exact sequence in local cohomology.

\begin{sdef}
An element $x \in R$ is \textbf{filter regular} or \textbf{$\mfm$-filter regular} if $x \in \mfm$ and $x \not \in \mfp$ for any $\mfp \in \assnot_R(R)$. A sequence $\uline{x}=x_1,\cdots,x_t$ is a \textbf{filter regular sequence} if $x_1$ is filter regular, $x_2+x_1R$ is filter regular in $R/x_1R$, and so on -- equivalently that $x_i \not \in \mfp$ for all $\mfp \in \assnot_R(R/(x_1,\cdots,x_{i-1}))$.
\end{sdef}

\begin{srmk}The sequence $\uline{x} = x_1,\cdots,x_t$ is filter regular if and only if $\uline{x}^{\uline{n}}=x_1^{n_1},\cdots,x_t^{n_d}$ is a filter regular sequence for any $\uline{n} \in (\mbz_+)^t$.\end{srmk}

\begin{sprop}
Let $\mfq\subset R$ be a parameter ideal. Then, there is a filter regular system of parameters $\uline{x}=x_1,\cdots,x_d$ such that $\mfq=(\uline{x})$.
\end{sprop}

\begin{proof}
If $d=0$, there is nothing to prove. Otherwise, pick the first parameter:\[x_1 \in \mfq \setminus \displaystyle \left(\mfm\mfq \cup \bigcup_{\mfp \in \assnot_R(R)} \mfp \right),\] which is a nonempty set by prime avoidance. Then we repeat in $R/x_1R$. After we have selected $d$ such elements, we have $d$ minimal generators $\uline{x}=x_1,\cdots,x_d$ of $\mfq$, a parameter ideal, so $(\uline{x})=\mfq$.
\end{proof}

\begin{sprop}\label{filter regular les}
Let $I\subset R$ be an ideal and $\uline{x}=x_1,\cdots,x_t$ a filter regular sequence in $R$ with $J=(\uline{x})$. (Recall $J_i = (x_1,\cdots,x_i)$.) Then $(J_{i-1}:_R x_i)/J_{i-1}$ is finite length over $R$ for each $1 \le i \le t$, so the short exact sequence: \begin{center}\begin{tikzcd} 
0 \arrow{r} & R/(J_{i-1}:x_i) \arrow{r}{\cdot x_i} & R/J_{i-1} \arrow{r}{\pi} & R/J_i \arrow{r} & 0 \end{tikzcd}\end{center} induces the following long exact sequence in local cohomology when $j>0$: \begin{center}
\begin{tikzpicture}
\node[scale=.96] at (0,0) {\begin{tikzcd}
\cdots \arrow{r} & H^j_I(R/J_{i-1}) \arrow{r}{\cdot x_i} & H^j_I(R/J_{i-1}) \arrow{r}{\alpha} & H^j_I(R/J_i) \arrow{r}{\beta} & H^{j+1}_I(R/J_{i-1})\arrow{r}{\cdot x_i} & \cdots
\end{tikzcd} };
\end{tikzpicture}
\end{center}
\end{sprop}

\begin{proof}
Since $x_i+J_{i-1} \not \in \mfp/J_{i-1}$ for any $\mfp \in \assnot_{R} (R/J_{i-1})$, this forces $\ass_R((J_{i-1}:_R x_i)/J_{i-1})\subset \{ \mfm\}$. Then since $(J_{i-1}:x_i)/J_{i-1}$ is finitely generated, it must be finite length. Consequently, $H^j_I((J_{i-1}:_R x_i)/J_{i-1})=0$ for any $j>0$, and so the short exact sequence: \begin{center}
\begin{tikzcd}
0 \arrow{r} & (J_{i-1}:_R x_i))/J_{i-1} \arrow{r} & R/J_{i-1} \arrow{r} & R/(J_{i-1}:_R x_i)) \arrow{r} & 0
\end{tikzcd}
\end{center} gives $H^j_I(R/J_{i-1}) \simeq H^j_I(R/(J_{i-1}:_R x_i)))$ for any $j>0$. Then, simply apply $H^j_I(\dot)$ to the short exact sequence in the statement of the proposition.
\end{proof}

\section{Frobenius actions on modules}

We now return to the prime characteristic case. Frobenius actions on Artinian modules and the dual theory of Cartier actions on finitely-generated modules have been studied extensively in recent literature. We will use the canonical Frobenius action on local cohomology to control the Frobenius test exponent of the ring.

\subsection{Basics of Frobenius actions}

\begin{sdef}
Let $M$ and $N$ be $R$-modules and let $\alpha: M \rightarrow N$ be an abelian group homomorphism. Say $\alpha$ is \textbf{$p^e$-linear} for some $e \in \mbn$ if $\alpha(xm)=x^{p^e}\alpha(m)$ for any $x \in R$ and $m \in M$. A $p$-linear endomorphism $f$ on $M$ is called a \textbf{Frobenius action on $M$}. If $M=R$, then there is a standard choice of $f$ -- the Frobenius endomorphism $F(r)=r^p$. Throughout this paper, any ring of prime characteristic will always be considered to have this choice of Frobenius action.
\end{sdef}

\begin{sdef}
Let $M$ be an $R$-module with a Frobenius action $f$. A submodule $M' \subset M$ is an \textbf{$f$-submodule} if $f(M')\subset M'$. If $M'$ is an $f$-submodule, we can define the \textbf{Frobenius orbit closure of $M'$} to be $(M')^f_M=\{m \in M \mid f^e(m) \in M' \text{ for some } e \in \mbn\}$. Clearly if $M'\subset M$ is an $f$-submodule, then $f|_{M'}: M' \rightarrow M'$ is a Frobenius action on $M'$.
\end{sdef}

\begin{srmk} There is also a notion of Frobenius closure for submodules generalizing the Frobenius closure of ideals in \Cref{frob closure defn} that is distinct from this sort of Frobenius closure (see \cite{PQ} for a definition). For instance, considering $I\subset R$, $I^F_R = \sqrt{I}$ whereas the Frobenius closure $I^F$ of $I$ is usually strictly smaller than $\sqrt{I}$.
\end{srmk}

\begin{sdef}
Let $M$ and $N$ be $R$-modules with Frobenius actions $f_M$ and $f_N$ respectively. An $R$-linear map $\alpha: M \rightarrow N$ \textbf{commutes with Frobenius} if $f_N \circ \alpha = \alpha \circ f_M$. 
\end{sdef}

\begin{sprop}\label{basics of frob actions}
Let $M$, $M'$, $N$, and $N'$ be $R$-modules with Frobenius actions, and let $\alpha: M \rightarrow N$ and $\beta: M' \rightarrow N'$ be maps which commute with Frobenius. Furthermore, let $N''\subset N$ be an $f$-submodule. Then:
\begin{enumerate}[label=\alph*)]
\item $\im(\alpha)$ and $\ker(\alpha)$ are $f$-submodules.
\item $N/N''$ has an unique Frobenius action such that the projection map $\pi: N \rightarrow N/N''$ commutes.
\item $\alpha(0^f_M)\subset 0^f_N$, and  moreover $\alpha^{-1}((N'')^f_{N}) = (\alpha^{-1}(N''))^f_M$. This implies $(N'')^f_N = \pi^{-1}\left(0^f_{N/N''}\right)$. 
\item $M \oplus M'$ has an induced Frobenius action commuting with inclusion and projection of the summands and furthermore, $\alpha \oplus \beta: M\oplus M' \rightarrow N \oplus N'$ commutes with this action. 
\end{enumerate}
\end{sprop}

\begin{sex}
Let $S$ be another characteristic $p$ ring and $\phi: R \rightarrow S$ be a ring homomorphism. Then $\phi$ commutes with Frobenius: $\phi(F(x))=\phi(x^p)=\phi(x)^p=F(\phi(x))$. 
\end{sex} 

\begin{sex}\label{frob action local cohomology}
The \v{C}ech cocomplex $\check{C}^j(\uline{x};R)$ on the elements $\uline{x}=x_1,\cdots,x_t\in R$ has the local cohomology modules for $(\uline{x})$ as its cohomology. Recall \[\check{C}^j: = \bigoplus_{1 \le i_1 < \cdots < i_j \le t} R_{x_{i_1}\cdots x_{i_j}}\] which has a natural Frobenius action by \Cref{basics of frob actions}. Furthermore, the maps in the cocomplex commute with these actions, and consequently each cohomology module $H^j_{(\uline{x})}(R)$ has a Frobenius action.\end{sex}

\subsection{Hartshorne-Speiser-Lyubeznik numbers}

Given part c) of \Cref{basics of frob actions}, to understand the Frobenius orbit closure of an $f$-submodule $M'\subset M$, it suffices to study the orbit closure of $0 \subset M/M'$, and only study elements which are ``nilpotent" under Frobenius. We have: \[ 0^f_M = \bigcup_{e \in \mbn} \ker(f^e: M \rightarrow M).\] It is natural to seek a single $e \in \mbn$ such that $0^f_M = \ker(f^e)$. 

\begin{sdef}\label{hsl numbers defn}
Let $M$ be an $R$-module with a Frobenius action. Define the \textbf{Hartshorne-Speiser-Lyubeznik number of $M$} to be: \[\hsl(M)= \inf\left\lbrace e \in \mbn \mid f^e\left(0^f_M\right) = 0\right\rbrace \in \mbn \cup \{\infty\}.\]
\end{sdef}

If $M$ is finitely generated, for any generating set $m_1,\cdots,m_r$ of $0^f_M$ we have for each $i$ an $e_i \in \mbn$ such that $f^{e_i}(m_i)=0$. Then $\hsl(M)\le \max{e_i}<\infty$. Another important case is also known.

\begin{sthm}[\cite{HS}, \cite{L}, \cite{S}]
Let $A$ be an Artinian $R$-module with a Frobenius action. Then $\hsl(A)<\infty$.
\end{sthm} 

\begin{srmk}
It is common to redefine $\hsl(R) = \max \{ 0 \le j \le d \mid \hsl(H^j_\mfm(R)\}$, which is finite by the previous theorem. Note this does not agree the notation in \Cref{hsl numbers defn} applied to $R$ -- however, we will not have need for the latter meaning here.

\end{srmk}
\subsection{The relative Frobenius action on local cohomology}

This section summarizes material from \cite{PQ} and \cite{Q}. 

\begin{sdef}
Let $I,J\subset R$ be ideals. The Frobenius endomorphism $F: R/J \rightarrow R/J$ can be factored as follows: \begin{center}
\begin{tikzcd}
R/J \arrow{rr}{F}\arrow{dr}{f_R} & \, & R/J \\
\, & R/J^{\fbp{p}} \arrow{ur}{\pi}& \,
\end{tikzcd}
\end{center} where $f_R(x+J) = x^p + J^{\fbp{p}}$. Call the map $f_R$ the \textbf{relative Frobenius map} on $R/J$. Note $f_R$ is $p$-linear. For any ideal $I\subset R$, the $p^e$-linear map $f_R^e: R/J \rightarrow R/J^{\fbp{p^e}}$ induces the \textbf{relative Frobenius action on local cohomology}, $f_R^e: H^j_I(R/J)\rightarrow H^j_I (R/J^{\fbp{p^e}} )$. 
\end{sdef}

\begin{srmk}
It is useful to note that when $J=0$, the diagram simply gives the standard Frobenius action on local cohomology. 
\end{srmk}

\begin{sdef}
Let $I,J\subset R$ be ideals. The \textbf{relative Frobenius closure of zero in $H^j_I(R/J)$} is the submodule: \[0^{f_R}_{H^j_I(R/J)} = \left\lbrace \xi \in H^j_I(R/J) \mid f^e_R(\xi) = 0 \in H^j_I\left(R/J^{\fbp{p^e}}\right) \text{ for some } e \in \mbn\right\rbrace.\] The \textbf{relative Hartshorne-Speiser-Lyubeznik number of $H^j_I(R/J)$} is: \[ \hslr(H^j_I(R/J))= \inf \left\lbrace e \in \mbn \mid f^e_R\left( 0^{f_R}_{H^j_I(R/J)}\right) = 0 \subset H^j_I\left(R/J^{\fbp{p^e}}\right)\right\rbrace \in \mbn \cup \{\infty\}.\]
\end{sdef}

There does not seem to be another definition of ``relative Frobenius closure" so we omit the descriptor ``orbit" used earlier. 

\begin{sprop}\label{filter regular les plus frob}
Given any ideal $I\subset R$ and any filter regular sequence $\uline{x}=x_1,\cdots,x_t$ with $J=(\uline{x})$, for any $e \in \mbn$ we have a commutative diagram with exact rows: \begin{center} \begin{tikzcd}
\cdots \arrow{r}{\cdot x_i} & H^j_I(R/J_{i-1}) \arrow{r}{\alpha_0} \arrow{d}{f^e_R} & H^j_I(R/J_i) \arrow{r}{\beta_0} \arrow{d}{f^e_R} & H^{j+1}_I(R/J_{i-1}) \arrow{r} \arrow{d}{f^e_R} & \cdots \\
\cdots \arrow{r}{\cdot x_i^{p^e}} & H^j_I\left(R/J_{i-1}^{\fbp{p^e}}\right) \arrow{r}{\alpha_e} & H^j_I\left(R/J_i^{\fbp{p^e}}\right) \arrow{r}{\beta_e} & H^{j+1}_I\left(R/J_{i-1}^{\fbp{p^e}}\right) \arrow{r} & \cdots 
\end{tikzcd}.
\end{center} for $j > 0$ and $1\le i\le t$, where the maps are those given in \Cref{filter regular les}.
\end{sprop}

\begin{proof}
Fix $e \in \mbn$ and consider the following commutative diagram: \begin{center}
\begin{tikzcd}
0 \arrow{r} & R/(J_{i-1}:_Rx_i) \arrow{r}{\cdot x_i} \arrow{d}{(f^e_R)'}& R/J_{i-1} \arrow{r}{\pi}\arrow{d}{f^e_R} & R/J_i \arrow{d}{f^e_R} \arrow{r} & 0 \\
0 \arrow{r} & R/\left(J_{i-1}^{\fbp{p^e}}:_Rx_i^{p^e}\right) \arrow{r}{\cdot x_i^{p^e}}  & R/J_{i-1}^{\fbp{p^e}} \arrow{r}{\pi_e} & R/J_i^{\fbp{p^e}} \arrow{r}  & 0 
\end{tikzcd},
\end{center} where we define $(f^e_R)': R/(J_{i-1}:x_i)\rightarrow R/( J_{i-1}^{\fbp{p^e}} :_R x_i^{p^e})$ by $x + (J_{i-1} :_R x_i)\mapsto x^{p^e} + (J_{i-1}^{\fbp{p^e}}:_R x_i^{p^e})$. Now apply the functor $H^j_I(\dot)$ and check that the map $(f_R)'$ composed with the isomorphism given in the proof of \Cref{filter regular les} gives $f_R$. 
\end{proof}

\begin{sprop}
Let $\uline{x}=x_1,\cdots,x_d$ be a filter regular system of parameters and $\mfq=(\uline{x})$. Then $\hslr(H^0_\mfm(R/\mfq)) = \fte(\mfq)$ (recall \Cref{ftedefn}).
\end{sprop}

\begin{proof}
Since $\mfq$ is $\mfm$-primary, $R/\mfq$ is $\mfm$-torsion, and hence $H^0_\mfm(R/\mfq)=R/\mfq$. Note the relative Frobenius action $f^e_R: R/\mfq \rightarrow R/\mfq^{\fbp{p^e}}$ recovers the Frobenius closure of 0, namely $0^{f_R}_{R/\mfq} = \mfq^F/\mfq$. Then, observe $f^e_R(\mfq^F/\mfq)=(\mfq^F )^{\fbp{p^e}}/\mfq^{\fbp{p^e}}$ which completes the proof.
\end{proof}

This connection and the diagram in \Cref{filter regular les plus frob} gives us the ability to use $\hsl(H^j_\mfm(R))$ to control $\fte(\mfq)$, as long as we have some control on what maps into the kernel of $f^e_R$. Although $f_R$ is not a true Frobenius action, the proof of \Cref{basics of frob actions} part c) still works, showing:  \[ \alpha_e\left( 0^{f_R}_{H^j_\mfm\left(R/\mfq_{i-1}^{\fbp{p^e}} \right)} \right) \subset 0^{f_R}_{H^j_\mfm\left(R/\mfq_i^{\fbp{p^e}}\right)}. \] 

\section{Finite Frobenius test exponents}
\subsection{The sufficient condition}

Quy essentially uses the condition given in \Cref{the condition} in the major theorems of \cite{Q}. Isolating this condition in particular allows us to expand the classes of rings known to have finite Frobenius test exponent.

\begin{sthm}\label{the condition}
Recall the notation in \Cref{filter regular les plus frob}, and specialize to the case that $I=\mfm$. Suppose there is an $e_0 \in \mbn$ depending only on the ring such that for any $e \ge e_0$ and any filter regular system of parameters $\uline{x}=x_1,\cdots,x_d$ with $\mfq=(\uline{x})$, we have: \[\alpha_{e}^{-1}\left(0^{f_R}_{H^j_\mfm\left(R/\mfq_i^{\fbp{p^{e}}}\right)}\right)= 0^{f_R}_{H^j_\mfm\left(R/\mfq_{i-1}^{\fbp{p^{e}}}\right)}\] for all $0 \le i+j < d$. Then, \[ \on{Fte}(R)\le e_0 +\sum^d_{k=0} \binom{d}{k} \on{HSL}(H^k_\mfm(R)).\]
\end{sthm}

\begin{proof}
Replace $\mfq$ by $\mfq^{\fbp{p^{e_0}}}$ and note $\fte(\mfq) \le \fte\left(\mfq^{\fbp{p^{e_0}}}\right)+e_0$, so it suffices to assume each $\alpha_e$ has the displayed property above. For notational convenience, we set $S_{i,e} = R/\mfq_i^{\fbp{p^e}}$ and $S_i = S_{i,0}$.

We now claim that: \[\on{HSL}_R(H^j_\mfm(S_i)) \le \sum_{k=j}^{i+j} \binom{i}{k-j}\on{HSL}(H^k_\mfm(R)) \] for all $i+j\le d$. We will show this by induction on $i$. If $i=0$, then $\hslr(H^j_\mfm(R)) = \hsl(H^j_\mfm(R))$, so there is nothing to show.

For our induction hypothesis, suppose \[\hslr(H^j_\mfm(S_{i-1})\le \sum_{k=j}^{i-1+j} \binom{i-1}{k-j} \hsl(H^k_\mfm(R))\] for all $0 \le j <d-(i-1)$. Let $e = \hslr(H^{j+1}_\mfm(S_{i-1}) )$ and $e'= \hslr(H^j_\mfm( S_{i-1,e} ))$. By the inductive hypothesis and manipulation of the binomial coefficients, $e+e' \le \sum_{k=j}^{i+j}\binom{i}{k-j} \hsl(H^k_\mfm(R))$ so it suffices to show that $\hslr(H^j_\mfm(S_i)) \le e+e'$.

Again for convenience, set $0^{f_R}_{H^j_\mfm(S_{i,e})} = 0^{f_R}_{j,i,e}$. Now, take $\xi \in 0^{f_R}_{j,i,0}$, so that $\beta_0(\xi)\in 0^{f_R}_{j+1,i-1,0}$. By our choice of $e$, we have: $0=f^e_R(\beta_0(\xi))=\beta_e(f^e_R(\xi))$ so that \[f^e_R(\xi) \in \ker(\beta_e) = \im(\alpha_e)\subset H^j_\mfm(S_{i,e}).\] Thus, there is an $\xi' \in H^j_\mfm(S_{i,e})$ such that $\alpha_e(\xi')=f^e_R(\xi) \in 0^{f_R}_{j,i,e}$. 
 
But $\alpha_e^{-1}(0^{f_R}_{j,i,e}) = 0^{f_R}_{j,i-1,e}$ by the condition on $\alpha_e$, thus by choice of $e'$ we have $f^{e'}_R(\xi')=0$ so that: \[0=\alpha_{e+e'}(f^{e'}_R(\xi'))=f^{e'}_R(\alpha_e(\xi')) = f^{e'}_R(f^e_R(\xi)) = f^{e+e'}_R(\xi).\] Since $\xi\in 0^{f_R}_{j,i,0}$ was arbitrary, the result is shown. 
\end{proof}

\subsection{An application}

As mentioned earlier in the paper, the following two classes of rings were known to have finite Frobenius test exponent.

\begin{sdef}\label{gcm defn}
Let $(R, \mfm)$ be a local ring of dimension $d$. $R$ is \textbf{generalized Cohen-Macaulay} if for all $0 \le j < d$, $H^j_\mfm(R)$ is finite length.
\end{sdef}

\begin{sdef}\label{wfn defn}
Let $(R,\mfm)$ be a local ring of dimension $d$ and of prime characteristic $p>0$. $R$ is \textbf{weakly F-nilpotent} if the standard Frobenius actions on $H^j_\mfm(R)$ are nilpotent, i.e. if $H^j_\mfm(R)=0^F_{H^j_\mfm(R)}$ for each $0 \le j < d$.\end{sdef}

We can mix these definitions together to establish a new class of F-singularity to which we can apply \Cref{the condition}.

\begin{sdef}\label{defn gwfnilp}
Say $R$ is \textbf{generalized weakly F-nilpotent} if $H^j_\mfm(R)/0^F_{H^j_\mfm(R)}$ is finite length for all $0 \le j < d$.
\end{sdef}

\begin{slem}\label{uniform fnilpotent multipliers}
Suppose $R$ is generalized weakly F-nilpotent. Then there is an $e_0 \in \mbn$ depending only on $R$ such that for all $e\ge e_0$ and any filter regular system of parameters $\uline{x} = x_1,\cdots,x_d$ with $\mfq=(\uline{x})$, we have: \[\mfq^{\fbp{p^e}}\cdot H^j_\mfm\left(R/\mfq_i^{\fbp{p^e}} \right)\subset 0^{f_R}_{H^j_\mfm\left(R/\mfq_i^{\fbp{p^e}} \right)},\] for all $0 \le i \le d-1$ and $0 \le j < d-i$.
\end{slem}

\begin{proof} 
By hypothesis, the ideals: \[\mfb_j=\ann_R\left(H^j_\mfm(R)/0^F_{H^j_\mfm(R)}\right)\] for $0 \le j < d$ are $\mfm$-primary or all of $R$, and hence $\mfb=\mfb_0\cdots\mfb_{d-1}$ is either $\mfm$-primary or all of $R$. 

Let $\uline{x}=x_1,\cdots,x_d$ be a filter regular system of parameters with $\mfq=(\uline{x})$ and $\mfq_i=(x_1,\cdots,x_i)$. Recall the notation in the proof of \Cref{the condition}: $S_{i,e} = R/\mfq_i^{\fbp{p^e}}$, with $S_i = S_{i,0}$ and $0^{f_R}_{H^j_\mfm(S_{i,e})} = 0^{f_R}_{j,i,e}$. 

We claim that $\mfb^{2^i}\subset \ann_R\left(H^j_\mfm(S_i)/0^{f_R}_{j,i,0} \right)$ for $0 \le i \le d$ and $0 \le j < d-i$. As before, we induce on $i$. If $i=0$, then there is nothing to show. When $i>0$, we can consider the commutative diagram with exact rows from \Cref{filter regular les plus frob}: \begin{center} 
\begin{tikzcd}
\cdots \arrow{r}{\cdot x_i} & H^j_\mfm(S_{i-1}) \arrow{r}{\alpha_0} \arrow{d}{f^e_R} & H^j_\mfm(S_i) \arrow{r}{\beta_0} \arrow{d}{f^e_R} & H^{j+1}_\mfm(S_{i-1}) \arrow{r} \arrow{d}{f^e_R} & \cdots \\
\cdots \arrow{r}{\cdot x_i^{p^e}} & H^j_\mfm(S_{i-1,e}) \arrow{r}{\alpha_e}  & H^j_\mfm(S_{i,e}) \arrow{r}{\beta_e} & H^{j+1}_\mfm(S_{i-1,e}) \arrow{r} & \cdots 
\end{tikzcd}.
\end{center} 

Let $\xi \in H^j_\mfm(S_i)$ and suppose $x,y \in \mfb^{2^{i-1}}$. But then $\beta_0(\xi) \in H^{j+1}_\mfm(S_{i-1})$ and so by hypothesis $x\beta_0(\xi)=\beta_0(x\xi) \in 0^{f_R}_{j+1,i-1,0}$, so there is an $e \in \mbn$ such that: \[ f^e_R(\beta_0(x\xi))=\beta_e(f^e_R(x\xi)) = 0. \] By exactness, there is a $\xi' \in H^j_\mfm(S_{i-1},e)$ such that $\alpha_e(\xi') = f^e_R(x\xi)$. 

But $y^{p^e} \in \mfb^{2^{i-1}}$ so by hypothesis $y^{p^e} \xi' \in 0^{f_R}_{j,i-1,e}$, and thus: \[ \alpha_e(y^{p^e}\xi') = y^{p^e}f^e_R(x\xi)=f^e_R(xy\xi) \in 0^{f_R}_{j,i,e}\] which implies $xy \xi \in 0^{f_R}_{j,i,0}$. Hence $xy \in \ann_R\left( H^j_\mfm(S_i)/0^{f_R}_{j,i,0}\right)$, proving the claim. Finally, pick $N$ minimal so $\mfm^N \subset \mfb$. Then for the smallest $e_0 \in \mbn$ with $p^{e_0} \ge N 2^{d-1}$, \[\mfq^{\fbp{p^e}} H^j_\mfm(S_{i,e}) \subset 0^{f_R}_{j,i,e} \] for any $e \ge e_0$, proving the lemma.
\end{proof}

\begin{sthm}\label{gwfncase}
Suppose $R$ is a generalized weakly F-nilpotent ring. Then $\fte(R)<\infty$. 
\end{sthm}

\begin{proof}
Adopt the notation in the proofs of \Cref{the condition} and \Cref{uniform fnilpotent multipliers}. Let $\mfq \subset R$ be any parameter ideal, and by replacing $\mfq$ with $\mfq^{\fbp{p^{e_0}}}$ as in the lemma, we may assume: \[ \mfq H^j_\mfm(S_i) \subset 0^{f_R}_{j,i,0} \] for any $0 \le i+j<d$.

Fix $0 \le i \le d-1$ and pick $e \in \mbn$ and $0 < j < d-i$. Then suppose $\alpha_e(\xi) \in 0^{f_R}_{j,i,e}$. For some $e' \in \mbn$, we have $f^{e'}_R(\xi)\in \ker(\alpha_{e+e'})=\im(x_i^{p^{e+e'}})$. By hypothesis, \[ \mfq^{\fbp{p^{e+e'}}} H^j_\mfm(S_{i,{e+e'}})\subset 0^{f_R}_{j,i,e+e'}\] But then $f^{e'}_R(\xi) \in 0^{f_R}_{j,i-1,e+e'}$ and hence $\xi \in 0^{f_R}_{j,i-1,e}$.

When $j=0$ we can exploit that $H^0_\mfm(R/J)$ is an ideal in $R/J$ for any ideal $J\subset R$. Note that the map $\cdot x_i: H^0_\mfm(R/(\mfq_{i-1}:x_i)) \rightarrow H^0_\mfm(S_{i-1})$ sends a class $r+(\mfq_{i-1}:x_i)$ to $x_ir + \mfq_{i-1} = x_i(r+\mfq_{i-1})$, so that if $x_i H^0_\mfm(S_i) \subset 0^{f_R}_{0,i,0}$, then \[x_i H^0_\mfm(R/(\mfq_{i-1}:x_i)) \subset 0^{f_R}_{0,i,0}.\] We have now shown the sufficient condition holds, so $\fte(R)<\infty$. 
\end{proof} 

\Cref{gwfncase} allows us to recapture the cases mentioned in the introduction. In particular, we have the following corollary.

\begin{scor}[\cite{HKSY}, \cite{Q}]\label{gcmcase}
Let $(R,\mfm)$ be a local ring of prime characteristic $p>0$. Then if $R$ is either generalized Cohen-Macaulay or weakly F-nilpotent, we have $\fte(R)<\infty$.
\end{scor}

\begin{proof}
In either case, observe $R$ is generalized weakly F-nilpotent, and apply \Cref{gwfncase}.
\end{proof}

\bibliography{ref}{} \bibliographystyle{alpha}

\begin{thebibliography}{HKSY06}

\bibitem[Bre06]{B}
Holger Brenner.
\newblock Bounds for test exponents.
\newblock {\em Compositio Mathematica}, 142(02):451--463, Mar 2006.

\bibitem[HKSY06]{HKSY}
Craig Huneke, Mordechai Katzman, Rodney~Y. Sharp, and Yongwei Yao.
\newblock {Frobenius} test exponents for parameter ideals in generalized
  {Cohen{\textendash}Macaulay} local rings.
\newblock {\em Journal of Algebra}, 305(1):516--539, Nov 2006.

\bibitem[HS77]{HS}
Robin Hartshorne and Robert Speiser.
\newblock {Local} {Cohomological} {Dimension} in {Characteristic} p.
\newblock {\em The Annals of Mathematics}, 105(1):45, Jan 1977.

\bibitem[KS06]{KS}
Mordechai Katzman and Rodney~Y. Sharp.
\newblock Uniform behaviour of the {Frobenius} closures of ideals generated by
  regular sequences.
\newblock {\em Journal of Algebra}, 295(1):231--246, Jan 2006.

\bibitem[Lyu97]{L}
Gennady Lyubeznik.
\newblock F-modules: applications to local cohomology and {D-modules} in
  characteristic p$>$0.
\newblock {\em Journal f\"{u}r die reine und angewandte Mathematik (Crelles
  Journal)}, 1997(491), 1997.

\bibitem[PQ18]{PQ}
Thomas Polstra and Pham~Hung Quy.
\newblock Nilpotence of {Frobenius} actions on local cohomology and {Frobenius}
  closure of ideals, 2018.
\newblock Preprint, arXiv:1803.04081.

\bibitem[QS16]{QS}
Pham~Hung Quy and Kazuma Shimomoto.
\newblock {$F$}-injectivity and {Frobenius} closure of ideals in {Noetherian}
  rings of characteristic $p>0$, 2016.
\newblock Preprint, arXiv:1601.02524.

\bibitem[Quy18]{Q}
Pham~Hung Quy.
\newblock On the uniform bound of {Frobenius} test exponents, 2018.
\newblock Preprint, arXiv:1804.01012.

\bibitem[Sha07]{S}
Rodney~Y. Sharp.
\newblock On the {Hartshorne-Speiser-Lyubeznik} theorem about {Artinian}
  modules with a {Frobenius} action.
\newblock {\em Proceedings of the American Mathematical Society},
  135(03):665--671, Mar 2007.

\end{thebibliography}
\end{document}